\documentclass[preprint,12pt]{elsarticle}

\usepackage{amssymb}
\usepackage{amsmath}
\usepackage{enumerate}
\usepackage{amsthm}
\usepackage{etoolbox}
\usepackage[ampersand]{easylist}
\usepackage{tikz}
\usepackage{caption}
\usepackage{subcaption}
\usepackage{pgfplots}
\usepackage{bondgraph-LAC}
\usepackage{import}
\usepackage{mechanical}
\usepackage{algorithm}
\usepackage{algpseudocode}

\usetikzlibrary{arrows.new}
\usetikzlibrary{%
    decorations.pathreplacing,%
    decorations.pathmorphing%
}
\usetikzlibrary{patterns}
\usetikzlibrary{calc,positioning}
\usetikzlibrary{shapes}
\tikzset{x=1cm,y=1cm}

\journal{arXiv}

\begin{document}

\theoremstyle{definition}
\newtheorem{definition}{Definition}

\newtheorem{theorem}{Theorem}
\newtheorem{proposition}{Proposition}
\newtheorem{lemma}{\it Lemma}
\newtheorem{corollary}{Corollary}

\theoremstyle{remark}
\newtheorem{remark}{Remark}

\renewcommand*{\arraycolsep}{2pt}

\newcommand{\mychi}{\raisebox{2pt}[1ex][1ex]{$\chi$}}

\begin{frontmatter}



\title{Localization of Control Synthesis Problem for Large-Scale Interconnected System Using IQC and Dissipativity Theories}


\author{Olugbenga Moses Anubi and Layne Clemen}

\address{Department of Mechanical and Aerospace Engineering,\\University of California,\\ Davis, California, USA}
\ead{anubimoses@gmail.com, laclemen@ucdavis.edu}

\begin{abstract}
The synthesis problem for the compositional performance certification of interconnected systems is considered. A fairly unified description of control synthesis problem is given using integral quadratic constraints (IQC) and dissipativity. Starting with a given large-scale interconnected system and a global performance objective, an optimization problem is formulated to search for admissible dissipativity properties of each subsystems. Local control laws are then synthesized to certify the relevant dissipativity properties. Moreover, the term localization is introduced to describe a finite collection of syntheses problems, for the local subsystems, which are a feasibility certificate for the global synthesis problem. Consequently, the problem of localizing the global problem to a smaller collection of disjointed sets of subsystems, called groups, is considered. This works looks promising as another way of looking at decentralized control and also as a way of doing performance specifications for components in a large-scale system.
\end{abstract}


\end{frontmatter}


\section{Introduction}\label{Introduction}
Due to constantly increasing systems requirements and complexity, developments in decentralized and compositional analyses and control have received a lot of research attention. In most approaches to decentralized control, sufficient conditions for the behavioral satisfaction of each component or subsystem are derived. A large-scale optimization problem is then formulated for the global objective with the local behavioral sufficient conditions as constraints. The resulting solution algorithm generally requires a solution for the local problems in the loop. As a result, the approach does not scale well for systems with non-uniform\footnote{this will make indexing very challenging - as the subsystems have to be handled individually and differently} local analyses. In order to overcome this limitation, a unified framework for describing system behavioral properties at both the subsystems and global level is inevitable. The use of IQCs seems appealing since they provide a way of describing relationships between processes evolving in a complex dynamical system, in a form that is convenient for analysis. 

IQCs were originally introduced as a way to characterize troublesome nonlinearities, time-variance, and uncertainties that plagued the application of linear control methods. Ever since, they have been used extensively in control\cite{fu1998robust,megretski1997system,apkarian2006iqc}, and presently gaining interests in optimization \cite{lessard2014analysis,nishihara2015general}, as a unified way of characterizing system behavior from input-output point-of-view. They are used extensively in this work to describe desired input-output system characteristics. 

The following issues are addressed in this paper:\vspace{-0.25cm}
	\subsection*{Unified description of control synthesis problems} This paper describes a way of looking at control synthesis problems as convex optimization problems involving parametrized integral quadratic constraints. This exposes a paradigm that can help transfer results between optimization and control. Certain properties of the IQCs that are a certificate for the BIBO stability of the resulting closed loop system are provided. Moreover, the feasibility and the solution of the synthesis problem for linear systems are described.
	\subsection*{Global admissibility condition} Given a global synthesis problem associated with an interconnected system and a finite collection of local synthesis problems associated with the local subsystems, the condition for which the feasibility of the local syntheses problems implies that of the global problem is derived. This condition is similar to existing results in literature \cite{meissen2014compositional,scorletti2001lmi,langbort2004distributed}, different only by the containment of an optimized parameter. Solution approach using the alternating direction method of multipliers is given. Also, corollaries under the assumptions of well-posedness and quadratic parametrization are also provided.
	\subsection*{Localization} The term \emph{localization} is introduced to describe the local syntheses problems for which the global admissibility condition holds for all values of the optimized parameter. Furthermore, the concept of \emph{closest localization} is introduced by including certain norm-like objective. It is then shown that the closest localization corresponds to the most relaxed local problem  given any interconnected system and associated global problem. In addition, group localization is introduced as a way of localizing the global problem among disjointed sets of subsystems. An alternating minimization algorithm is also given, without convergence analysis/proof, as one way of solving the group localization problem. This relies on well established results regarding the convergence of alternating minimization algorithms \cite{wang2008new,csisz1984information,chan2000convergence}.

\section{Notation}\label{Notation}
Throughout the paper, the following notations are used: $\mathbb{R}\text{ and }\mathbb{R}_+$ denotes the set of real numbers and positive real numbers respectively. $\mathbb{R}_+^n$ denotes the first orthant of the $n$-dimensional euclidean space. The set of all symmetric positive definite and positive semi-definite matrices are denoted by $\mathbb{S}_{++}\text{ and }\mathbb{S}_+$ respectively. The Euclidean norm of a vector $\textbf{x}\in\mathbb{R}^n$ is denoted by $\left\|\textbf{x}\right\|\triangleq\left(\textbf{x}^\top\textbf{x}\right)^{1/2}$. The quadratic form $\left\|\textbf{x}\right\|_P^2\triangleq\textbf{x}^\top P\textbf{x}$ is defined for any matrix $P\in\mathbb{S}_+$. The expression $P\preceq Q$ means that the matrix $Q-P\in\mathbb{S}_+$. The set of the eigenvalues of a square matrix $A$ is denoted by $\lambda(A)$ with $\lambda_{max}(A)\triangleq\max\{\lambda(A)\}$ and $\lambda_{min}(A)\triangleq\min\{\lambda(A)\}$. Similarly, The set of the singular values of a matrix $A$ is denoted by $\sigma(A)$ with $\sigma_{max}(A)\triangleq\max\{\sigma(A)\}$ and $\sigma_{min}(A)\triangleq\min\{\sigma(A)\}$. Occasionally, the block symmetric matrix $\left[\begin{array}{c|c}A&B\\\hline B^{\top^{\phantom{T}}}&C\end{array}\right]$ is written as $\left[\begin{array}{c|c}A&B\\\hline *&C\end{array}\right]$ for compactness purpose.

Let $U\in\mathbb{R}^{k\times n}, \hspace{2mm} r<n$ be a full rank matrix, $U^\perp$ denotes the orthogonal complement of $U$, i.e. $UU^\perp=0$ and $[U\hspace{2mm}U^\perp]$ is of maximal rank.

The euclidean balls $\mathbb{B}_r(0)\text{ and }\mathbb{B}_r(\textbf{x}_0)$ are defined respectively for some $r\in\mathbb{R}_+$ as $\mathbb{B}_r(0)\triangleq\left\{\textbf{x}:\left\|\textbf{x}\right\|\le r\right\}$ and $\mathbb{B}_r(\textbf{x}_0)\triangleq\left\{\textbf{x}:\left\|\textbf{x}-\textbf{x}_0\right\|\le r\right\}$.

The space of all square-integrable signals $\textbf{f}:\mathbb{R}_+\rightarrow\mathbb{R}^n$ satisfying
\begin{align*}
\int_0^\infty\|\textbf{f}(t)\|^2dt<\infty
\end{align*}
is denoted by $\mathbb{L}_2^n$. Consequently, the $\mathcal{L}_2$-norm of a signal $\textbf{f}\in\mathbb{L}_2^n$ is defined as
\begin{align*}
\left\|\textbf{f}\right\|_2\triangleq\left(\int_0^\infty{\|\textbf{f}(t)\|^2dt}\right)^{\frac{1}{2}}
\end{align*}
Let $\textbf{f}:\mathbb{R}_+\mapsto\mathbb{R}^n$. Then for each $T\in\mathbb{R}_+$, the function $\textbf{f}_T:\mathbb{R}_+\mapsto\mathbb{R}^n$ is defined by
\begin{align*}
\textbf{f}_T(t)\triangleq\left\{\begin{array}{rl}\textbf{f}(t),&0\le t<T\\\textbf{0},&t\ge T\end{array}\right.
\end{align*}
and is called the \emph{truncation} of $\textbf{f}$ to the interval $[0,\hspace{2mm}T]$. Consequently, the set $\mathbb{L}_{2e}^n$ of all measurable signal $\textbf{f}:\mathbb{R}_+\mapsto\mathbb{R}^n$ such that $\textbf{f}_T(t)\in\mathbb{L}_2$ for all $T\in[0,\hspace{2mm}\infty)$ is called the extension of $\mathbb{L}_2^n$ or the extended $\mathbb{L}_2$-space. The Fourier transform of $\mathbf{f}\in\mathbb{L}_2^n$ is denoted by
\begin{align*}
\widehat{\mathbf{f}}(j\omega)\triangleq\int_0^\infty e^{-j\omega t}\mathbf{f}(t)dt.
\end{align*}

\section{Preliminaries}\label{Preliminary}


This section develops some preliminary results required for the main developments and applications in the subsequent sections.
\begin{definition}[Causality]
Let $P_T$ be a past projection operator\footnote{Operator is used to describe a function from one $\mathbb{L}_{2e}$ space to another} defined for any $T>0$ by
\begin{align}
P_T\mathbf{f}(t) = \mathbf{f}_T(t).
\end{align}
The operator $\Delta:\mathbb{L}_{2e}^p\mapsto\mathbb{L}_{2e}^r$ is then said to be causal if $P_T\Delta = P_T\Delta P_T$ for any $T>0$.
\end{definition}

\begin{definition}[BIBO\footnote{bounded-input bounded-ouput} Stability]
A causal operator $\Delta:\mathbb{L}_{2e}^p\mapsto\mathbb{L}_{2e}^r$ is said to be stable if there exists $c\in\mathbb{R}_+$ such that
\begin{align}
\int_0^T{\Delta(\mathbf{w}(t))^\top\Delta(\mathbf{w}(t)) dt}\le c\int_0^T{\mathbf{w}(t)^\top\mathbf{w}(t) dt}
\end{align}
for all $T\ge0$ and $\mathbf{w}\in\mathbb{L}_{2e}^p$.
\end{definition}
\subsection{Integral Quadratic Constraints (IQC)}

In general, IQCs give useful characterization of the structure of operators - which in this case are systems described by ordinary differential equations. 
\begin{definition}[Integral Quadratic Constraints\cite{megretski1997system}]
A bounded operator $\Delta:\mathbb{L}_{2e}^p\rightarrow\mathbb{L}_{2e}^r$ is said to satisfy the IQC defined by $\Pi$ if
\begin{align}\label{IQC1}
\int_{-\infty}^\infty\left[\begin{array}{c}\widehat{\mathbf{w}}(j\omega)\\\widehat{\mathbf{z}}(j\omega)\end{array}\right]^*\Pi(j\omega)\left[\begin{array}{c}\widehat{\mathbf{w}}(j\omega)\\\widehat{\mathbf{z}}(j\omega)\end{array}\right]d\omega\ge0
\end{align}
holds for all $\mathbf{z}=\Delta(\mathbf{w})$, $\mathbf{w}\in\mathbb{L}_2$.
\end{definition}
Often, $\Pi$ is referred to as the \emph{multiplier} that defines IQC and the shorthand notation $\Delta\in \text{IQC}(\Pi)$ is used to meean that $\Delta$ satisfies the IQC defined by $\Pi$.

\begin{remark}
This input-output characteristic will be used later to generalize the class of control synthesis problems considered in this paper.
\end{remark}
\begin{remark}
Observer that $\Delta\in \text{IQC}(\Pi)$ implies that $\Delta\in \text{IQC}(\alpha\Pi)$ for all $\alpha\in\mathbb{R}_+$. Thus, the set of multipliers $\mathcal{K}(\Delta) \triangleq \left\{\Pi|\hspace{2mm}\Delta \in\text{IQC}(\Pi)\right\}$ forms a convex cone. This is an indication that control synthesis problems for a dynamical system described by the operator $\Delta$ can be described as corresponding optimization problems over the cone $\mathcal{K}(\Delta)$.
\end{remark}
\begin{proposition}[Stability Certificate]
The operator $\Delta:\mathbb{L}_{2e}^p\mapsto\mathbb{L}_{2e}^r$ is stable if and only if there exist a bounded\footnote{bounded eigenvalues} $\Pi(j\omega)$ given by the conformal block
\begin{align}
\Pi(j\omega) = \left[\begin{array}{cc}\Pi_{11}(j\omega)&\Pi_{12}(j\omega)\\\Pi_{12}(j\omega)^*&\Pi_{22}(j\omega)\end{array}\right].
\end{align}
satisfying $\Pi_{11}(j\omega)\succeq0$, and $\Pi_{22}(j\omega)\preceq0$ for all $\omega\in\mathbb{R}$,
such that 
\begin{align}
\Delta\in \text{IQC}(\Pi).
\end{align}
\end{proposition}
\begin{proof}($\Rightarrow$) Suppose $\Delta$ is stable. Then, there exists $c\in\mathbb{R}_+$ such that
\begin{align*}
\int_0^\infty{\Delta(\mathbf{w}(t))^\top\Delta(\mathbf{w}(t)) - c\mathbf{w}(t)^\top\mathbf{w}(t) dt}\le 0,
\end{align*}
which, after using Perseval's identity, implies that 
\begin{align*}
\Delta\in \text{IQC}\left(\left[\begin{array}{cc}cI&0\\0&-I\end{array}\right]\right)
\end{align*}
($\Leftarrow$) Let $\pi_{ij}\triangleq\sup_w\left\|\Pi_{ij}(j\omega)\right\|_2,\hspace{2mm}i,j\in\{1,2\}$. Then, for all $\omega\in\mathbb{R}_+$,
\begin{align*}
\int_{-\infty}^\infty{\left[\begin{array}{c}\widehat{\mathbf{w}}(j\omega)\\\widehat{\Delta}(\mathbf{w}(j\omega))\end{array}\right]^*\Pi(j\omega)\left[\begin{array}{c}\widehat{\mathbf{w}}(j\omega)\\\widehat{\Delta}(\mathbf{w}(j\omega))\end{array}\right] d\omega}\ge0
\end{align*}
implies, using the young's inequality, that
\begin{align*}
\left(\pi_{11}+\varepsilon\right)\left\|\mathbf{w}\right\|_2^2 - \left(\pi_{22}-\frac{\pi_{12}^2}{\varepsilon}\right)\left\|\Delta(\mathbf{w})\right\|_2^2\ge0,
\end{align*}
for some $\varepsilon>0$. Choose $\displaystyle \varepsilon\ge\frac{\pi_{12}^2}{\pi_{22}}$, and since $\mathbf{w}\in\mathbb{L}_{2e}^p$, it follows that
\begin{align}
\int_0^T{\Delta(\mathbf{w}(t))^\top\Delta(\mathbf{w}(t)) dt}\le c\int_0^T{\mathbf{w}(t)^\top\mathbf{w}(t) dt}
\end{align}
for all $T\ge0$, where
\begin{align}\displaystyle
0\le c\triangleq \frac{\left(\pi_{11}+\varepsilon\right)\varepsilon}{\pi_{22}\varepsilon-\pi_{12}^2}
\end{align}
\end{proof}

\subsection{Dissipativity}
Consider a continuous-time, time-invariant dynamical system described by
\begin{align}\label{dyn1}
\Sigma :\left\{ \begin{array}{rclc}\dot{\textbf{x}}(t)&=&\textbf{f}(\textbf{x}(t),\textbf{w}(t)),&\hspace{3mm}\textbf{f}(\textbf{0},\textbf{0})=\textbf{0}\\
\textbf{z}(t)&=&\textbf{h}(\textbf{x}(t),\textbf{w}(t)),&\hspace{3mm}\textbf{h}(\textbf{0},\textbf{0})=\textbf{0}
\end{array}\right.
\end{align}
with $\textbf{x}\in\mathbb{X}\subset\mathbb{R}^n$, $\textbf{w}\in\mathbb{W}\subset\mathbb{R}^p$ and $\textbf{z}\in\mathbb{Z}\subset\mathbb{R}^r$.

\begin{definition}[Supply Rate]
Given the dynamical system in \eqref{dyn1}, a \emph{supply rate} is any mapping $s:\mathbb{W}\times\mathbb{Z}\mapsto\mathbb{R}$ satisfying
\begin{align}
\int_{t_0}^{t_1}|s(\textbf{w}(t),\textbf{z}(t))|dt<\infty
\end{align}
for all $t_0,t_1\in\mathbb{R}_+$ and for all input-output pair $(\textbf{w}(t),\textbf{z}(t))$ satisfying \eqref{dyn1}.
\end{definition}
Henceforth explicit time arguments in signals will be dropped, except otherwise needed for clarity. 
\begin{definition}[Dissipativity]
The system $\Sigma$ is said to be \emph{dissipative} with respect to a supply rate $s$ if there exists a differentiable and nonnegative function $V:\mathbb{X}\mapsto\mathbb{R}_+$ such that
\begin{align}\label{dissineq1}
\nabla V(\textbf{x})^\top\textbf{f}(\textbf{x},\textbf{w})-s(\textbf{w},\textbf{z})\le0
\end{align}
for all $\textbf{x}\in\mathbb{X}$, $\textbf{w}\in \mathbb{W}$. 
\end{definition}
\begin{remark} The inequality in \eqref{dissineq1} is referred to as the \emph{Dissipation Inequality} and describes a property of the system that stipulates that within any time interval, the change in the internal stored energy cannot exceed the total externally supplied energy. Hence, there can be no internal ``creation of energy"; only internal \emph{dissipation} is possible. If \eqref{dissineq1} holds with equality for all $\textbf{x}\in\mathbb{X}$, $\textbf{w}\in \mathbb{W}$, then $\Sigma$ is said to be \emph{lossless} with respect to $s$. 
\end{remark}
\begin{remark} It can be seen that IQC generalize the dissipativity framework to supply rates that are themselves dynamical systems by allowing frequency-dependent requirements to be described. Consequently, a definition of IQC in terms of dissipativity can be given. In the line of the definition given in \cite{meissen2014compositional}, let $(\bar{A},\bar{B},\bar{C},\bar{D})$ be a realization of a stable LTI system $\Psi$ with state vector $\boldsymbol{\eta}\in\mathbb{R}^{n_2}$ such that the multiplier $\Pi$ admits the factorization $\Pi=\Psi^*X\Psi$  . Then the system described by \eqref{dyn1} is said to satisfy the IQC defined by $\Pi$ if there exists a $\mathbb{C}^1$ mapping $V:\mathbb{R}^n\times\mathbb{R}^{n_2}\rightarrow\mathbb{R}_+$ such that
\begin{align}\nonumber
\nabla_xV(\textbf{x},\boldsymbol{\eta})^\top\textbf{f}(\textbf{x},\textbf{w})+\nabla_\eta V(\textbf{x},\boldsymbol{\eta})^\top\left(\bar{A}\boldsymbol{\eta}+\bar{B}\left[\begin{array}{c}\textbf{w}\\\textbf{z}\end{array}\right]\right)\le\hspace{2.5cm}\\\label{IQC2}
\left(\bar{C}\boldsymbol{\eta}+\bar{D}\left[\begin{array}{c}\textbf{w}\\\textbf{z}\end{array}\right]\right)^\top X\left(\bar{C}\boldsymbol{\eta}+\bar{D}\left[\begin{array}{c}\textbf{w}\\\textbf{z}\end{array}\right]\right)
\end{align}
for all $\textbf{x}\in\mathbb{X},\textbf{w}\in\mathbb{W}\text{ and }\textbf{z}=\textbf{h}(\textbf{x},\textbf{w})$.
Traditional dissipativity definition is recovered in the special case where $\Psi$ is static.
\end{remark}

\section{Synthesis Problem}\label{Synthesis_Problem}
Throughout the remainder of this paper, we shall limit our attention to static supply rates. Generalizing to dynamic supply rates is left for future work. If the static supply rate is parameter dependent, a fairly general control synthesis problem can then be defined.
\begin{figure}[!ht]
\centering
\begin{tikzpicture}[line width = 1.5pt,auto,>=latex,node distance=1.6cm]
\tikzstyle{block}=[draw,rectangle,minimum width=4em, minimum height=4em]
\def\x{1.0}		
\def\y{0.5}		
\def\z{1.4}		
\node [block](H){\Large $H$};
\node [draw,rectangle, minimum width=1em, minimum height=1em,thick](X) at (1.3em,-1.3em){$\mathbf{x}$};
\draw [->] (H.east)+(0,0) -- +(\z,0) node [anchor=west]{$\mathbf{y}$};
\draw [<-] (H.west)+(0,-\y) -- +(-\z,-\y) node [anchor=east]{$\mathbf{v}$};
\draw [<-] (H.west)+(0,+\y) -- +(-\z,+\y) node [anchor=east]{$\mathbf{u}$};
\end{tikzpicture}
\caption{Dynamic Operator}
\label{fig_synthesis}
\end{figure}
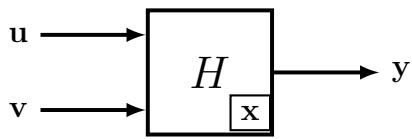

Consider the causal operator $H:\mathbb{L}_{2e}^{n_u}\times\mathbb{L}_{2e}^{n_v}\mapsto\mathbb{L}_{2e}^{n_y}$ as shown in Fig.~\ref{fig_synthesis}. Here, $\mathbf{x}\in\mathbb{R}^{n_x}$ represents the internal states of the operator, $\left(\mathbf{v},\mathbf{y}\right)\in\mathbb{L}_{2e}^{n_v}\times\mathbb{L}_{2e}^{n_y}$ are the input-output pair over which the desired performance objectives are defined and $\mathbf{u}:\mathbb{L}_{2e}^{n_x}\times\mathbb{L}_{2e}^{n_y}\mapsto\mathbb{L}_{2e}^{n_u}$ is the control input which is to be designed to achieve the performance objectives. The numbers  $n_x, n_v+n_y=n,n_u\in\mathbb{N}$ are the dimensions of the respective spaces. Now, we have all the ingredients to give a formal definition of a synthesis problem.

\begin{definition}[Synthesis problem]
Given a parametrized \emph{static} multiplier\footnote{the terms multiplier and supply rate are used interchangeably} $X(\gamma)\in\mathbb{R}^{n\times n},X_{11}(\gamma)\succeq0,X_{22}(\gamma)\preceq0$, with parameter $\gamma\in\mathbb{R}$. The multiplier $X(\gamma)$ is monotonically increasing in the parameter $\gamma$, namely $X(\gamma_1)-X(\gamma_2)\preceq0\Leftrightarrow\gamma_1-\gamma_2\le0$. The synthesis problem for the operator $H:\mathbb{L}_{2e}^{n_u}\times\mathbb{L}_{2e}^{n_v}\mapsto\mathbb{L}_{2e}^{n_y}$ is given by the optimization problem
\begin{align}\label{synthesis}
P_H:\left\{
\begin{array}{rl}
\text{min}&\gamma^2\\
\text{s.t}&H\in\text{IQC}(X(\gamma)),
\end{array}
\right.
\end{align}
and is denoted by $P_H(X(\gamma);\gamma^2)$
\end{definition}
\begin{remark}[Feasibility]
The synthesis problem $P_H(X(\gamma);\gamma^2)$ is said to be feasible if the associated optimization problem is feasible. i.e there exists a control law $\mathbf{u}:\mathbb{L}_{2e}^{n_x}\times\mathbb{L}_{2e}^{n_y}\mapsto\mathbb{L}_{2e}^{n_u}$ and at least one parameter $\gamma_f\in\mathbb{R}$ such that $H\in\text{IQC}(X(\gamma_f))$.
\end{remark}
\begin{remark}[Solution]
The shorthand $\gamma^* = P_H(X(\gamma);\gamma^2)$ is used to mean that $\gamma^*$ is the solution of the synthesis problem $P_H(X(\gamma);\gamma^2)$ and is defined as the solution of the corresponding optimization problem $P_H$ in \eqref{synthesis}. 
\end{remark}
\begin{remark}[Properties]
The following properties of the synthesis problem are obvious;

\begin{itemize}
	\item $P_H(X(\gamma);\gamma^2) = P_H(\alpha X(\gamma);\gamma^2)$ for all $\alpha\in\mathbb{R}_+$
	\item $P_H(X(\gamma);\gamma^2) = P_H(X(\gamma) + X^\top(\gamma);\gamma^2)$
\end{itemize}
Hence, without loss of generality, the supply rates (multipliers) will be assumed symmetric. The above properties also demonstrates that scaling is inconsequential for the solution of the synthesis problem. Only the eigenstructure of the multipliers are important.
\end{remark}
\subsection{Basic Synthesis Problems}
Here, we discuss the synthesis problems corresponding to some basic control objectives which can be described within the dissipativity framework. 
Consider the dynamical system
\begin{align}\label{dyn2}
H :\left\{ \begin{array}{rcl}\dot{\textbf{x}}&=&\textbf{f}(\textbf{x},\textbf{u},\textbf{v}),\\
\textbf{y}&=&\textbf{h}(\textbf{x},\textbf{u},\textbf{v}).
\end{array}\right.
\end{align}
The corresponding synthesis problem is given by
\begin{align}
P_H:\left\{\begin{array}{rl}
\text{min}&\gamma^2\\
\text{s.t}\\&\nabla V(\mathbf{x})^\top\mathbf{f}(\mathbf{x},\mathbf{u},\mathbf{v})-\left[\begin{array}{c}\mathbf{v}\\\mathbf{y}\end{array}\right]^\top X(\gamma)\left[\begin{array}{c}\mathbf{v}\\\mathbf{y}\end{array}\right]\le0,\\&\text{for all }\textbf{x}\in\mathbb{X},\textbf{v}\in\mathbb{L}_{2e}^{n_v}\text{ and }\textbf{y}=\textbf{h}(\textbf{x},\textbf{u},\mathbf{v}).
\end{array}\right.
\end{align}

\subsubsection{Square Systems}
The term \emph{square} is used to describe systems with equally dimensioned input and output. First, the input-output objectives considered are defined. Then, the synthesis problem encompassing the defined objectives is given.
\begin{definition}[Positive-real or Passivity Objectives]
The state space system $H$ in \eqref{dyn2} is:
\begin{itemize}
	\item \emph{passive} if it is dissipative with respect to the supply rate $s(\textbf{w},\textbf{z})=\textbf{w}^\top\textbf{z}$.
\item \emph{strictly output passive} if there exists $\varepsilon>0$ such that $H$ is dissipative with respect to the supply rate $s(\textbf{w},\textbf{z})=\textbf{w}^\top\textbf{z}-\varepsilon\left\|\textbf{z}\right\|^2$.
\end{itemize}
\end{definition}

\begin{definition}[Synthesis Problem for Positive-real Objectives]
The synthesis problem for positive-real objectives is given by the optimization problem
\begin{align}\label{synthesis_square}
P_H:\left\{
\begin{array}{rl}
\text{min}&-\varepsilon\\
\text{s.t}\\&H\in\text{IQC}\left(\left[\begin{array}{cc}0&I\\I&-\varepsilon I\end{array}\right]\right),\\&\hspace{2mm}\varepsilon\ge0,
\end{array}
\right.
\end{align}
\end{definition}
\begin{remark}
It is also possible to be interested in just the feasible problem. In which case, the objective of the synthesis problem is to certify a strict output passivity for a given $\varepsilon\ge0$. The corresponding synthesis problem in this case is given by $P_H\left(\left[\begin{array}{cc}0&I\\I&-\varepsilon I\end{array}\right];0\right)$.
\end{remark}
\subsubsection{Rectangular Systems}
Next, we consider systems in which input and output dimensions are not necessarily the same.
\begin{definition}[$\mathcal{L}_2$-gain]
The $\mathcal{L}_2$-gain of the state space system $H$ in \eqref{dyn2}, from the exogenous input $\mathbf{v}$ to the regulated output $\mathbf{y}$, is given by
\begin{align}
\gamma(H)\triangleq\inf\left\{\tilde{\gamma}|\exists b \ni \|\textbf{y}_T\|_2^2\le\tilde{\gamma}^2\|\textbf{v}_T\|_2^2+b,\hspace{2mm}\forall\hspace{2mm} \textbf{v}\in\mathbb{L}_{2e}^p,\hspace{2mm}\forall T\ge0\right\}
\end{align}
 \end{definition}
\begin{proposition}[Section 3.1 \cite{schaft1999l2}]
The state space system $\Sigma$ in \eqref{dyn1} has $\mathcal{L}_2$-gain$\le\gamma$ if it is dissipative with respect to the supply rate $s(\textbf{w},\textbf{z})=\frac{1}{2}\gamma^2\|\textbf{w}\|^2-\|\textbf{z}\|^2$.
\end{proposition}

Sometimes, it is desired to keep the system trajectory within a known set. That is, to enforce some constraints on the reachable set of the system. These constraints could translate directly to requirements such as safety, behavioral or validity. The following theorem describes the condition which provides a guarantee for the invariance of some sets under norm-boundedness assumptions on the exogenous input. This type of objective is termed \emph{reachability} because of the direct constraints imposed on the set of reachable states.

\begin{proposition}[Reachability]
Suppose $\textbf{f}:\mathbb{X}\times\mathbb{W}\mapsto\mathbb{R}^n$ in \eqref{dyn1} is Lipschitz continuous in both arguments. Given $\beta>0$ and a differentiable $V:\mathbb{X}\mapsto\mathbb{R}$ satisfying $V(\textbf{0})\le\beta$. Suppose that $\left\|\textbf{w}\right\|_2^2\le\beta$. If
\begin{align}\label{reach1}
\nabla V(\textbf{x})^\top\textbf{f}(\textbf{x},\textbf{w})\le\textbf{w}^T\textbf{w}
\end{align}
for all $\textbf{x}\in\mathbb{X}$ and $\textbf{w}\in\mathbb{W}$, then $V(\textbf{x})\le2\beta$ for all $t\in\mathbb{R}_+$ - meaning that the set $\mathcal{E}\triangleq\left\{\textbf{x}|\hspace{2mm}V(\textbf{x})\le2\beta \right\}$ is invariant
\end{proposition}
\begin{proof}
Integrating the dissipation inequality in \eqref{reach1} yields
\begin{align*}
V(\textbf{x}(t))&\le V(\textbf{0}) + \int_0^t{\|\textbf{w}(\tau)\|^2d\tau}\\
&\le2\beta.
\end{align*}
\end{proof}

\begin{definition}[Synthesis Problem for System-Norm Objectives]
The synthesis problem for norm objectives is given by
\begin{align}\label{synthesis_rectangle}
P_H:\left\{
\begin{array}{rl}
\text{min}&\gamma^2+\lambda\beta^2\\
\text{s.t}\\&\nabla V(\mathbf{x})^\top\mathbf{f}(\mathbf{x},\mathbf{u},\mathbf{v})-\left[\begin{array}{c}\mathbf{v}\\\mathbf{y}\end{array}\right]^\top\left[\begin{array}{cc}\gamma^2 I&0\\0&-I\end{array}\right]\left[\begin{array}{c}\mathbf{v}\\\mathbf{y}\end{array}\right]\le0,\\
&\left\|\mathbf{y}\right\|^2-\beta^2V(\mathbf{x})\le0,\\\\
&\text{for all }\textbf{x}\in\mathbb{X},\textbf{v}\in\mathbb{L}_{2e}^{n_v}\text{ and }\textbf{y}=\textbf{h}(\textbf{x},\textbf{u},\mathbf{v}),
\end{array}
\right.
\end{align}
where $\lambda\ge0$ is a weighting on the multi-objective performance index.
\begin{remark}
The first inequality constraint, together with the first term in the performance index, corresponds to the $\mathcal{L}_2$-gain objective. The second inequality and the second term of the performance index define an $H_2$-like objective for the synthesis problem. This is obvious when the system $H$ is linear time invariant with a Lyapunov function $V(\mathbf{x})=\mathbf{x}^TP\mathbf{x}$.
\end{remark}
\begin{remark}
Subsets of the above synthesis problem can be considered. For instance, removing the second inequality together with second term of the performance index results in the $H_\infty$ problem $P_H\left(\left[\begin{array}{cc}\gamma^2 I&0\\0&-I\end{array}\right];\gamma^2\right)$
\end{remark}
\end{definition}
\subsection{Linear Systems}
The goal is here to give a full description of the synthesis problem for linear systems. This includes the derivation of the feasibility conditions as well. Consider the linear system
\begin{align}
H_L:\left\{\left[\begin{array}{c}\dot{\mathbf{x}}\\\mathbf{y}\\\mathbf{y}_m\end{array}\right] = \left[\begin{array}{c|cc}A&B_1&B_2\\\hline C_1&D_{11}&D_{12}\\ C_2&D_{21}&0\end{array}\right]\left[\begin{array}{c}\mathbf{x}\\\mathbf{v}\\\mathbf{u}\end{array}\right],\right.
\end{align}
where $\mathbf{x}\in\mathbb{R}^n$ is the state vector, $\mathbf{y}\in\mathbb{R}^{n_y}$ is the output of interest, $\mathbf{y}_m\in\mathbb{R}^{n_m}$ is the measured output, $\mathbf{v}\in\mathbb{R}^{n_v}$ is the exogenous disturbance, and $\mathbf{u}\in\mathbb{R}^{n_u}$ is the control input. 

The controller considered is a finite dimensional LTI system described as
\begin{align}
\mathcal{C}:\left\{\left[\begin{array}{c}\dot{\mathbf{x}}_c\\\mathbf{u}\end{array}\right] = \left[\begin{array}{c|c}A_c&B_c\\\hline C_c&D_c\end{array}\right]\left[\begin{array}{c}\mathbf{x}_c\\\mathbf{y}_m\end{array}\right]\right.,
\end{align}
where $A_c,B_c,C_c,D_c$ are appropriately dimensioned parameters to be found. The controller input $\mathbf{y}_m$ and the output $\mathbf{u}$ matches the measured output and control input respectively of the LTI system $H_L$.

Consequently the closed loop system is given by
\begin{align}
H_{cl}:\left\{\left[\begin{array}{c}\dot{\mathbf{x}}_{cl}\\\mathbf{y}\end{array}\right] = \left[\begin{array}{c|c}A_{cl}&B_{cl}\\\hline C_{cl}&D_{cl}\end{array}\right]\left[\begin{array}{c}\mathbf{x}_{cl}\\\mathbf{v}\end{array}\right]\right.,
\end{align}
where

\begin{align}
\left[\begin{array}{c|c}A_{cl}&B_{cl}\\\hline C_{cl}&D_{cl}\end{array}\right] = \left[\begin{array}{cc|c}A&0&B_1\\0&0&0\\\hline C_1&0&D_{11}\end{array}\right] + \left[\begin{array}{cc}0&B_2\\I&0\\\hline 0&D_{12}\end{array}\right]\left[\begin{array}{c|c}A_{c}&B_{c}\\\hline C_{c}&D_{c}\end{array}\right]\left[\begin{array}{cc|c}0&I&0\\C_2&0&D_{21}\end{array}\right].
\end{align}
Given a static multiplier $X=X^\top$, it is straightforward to see that the condition $H_{cl}\in\text{IQC}(X)$ is equivalent to the existence of a matrix $P\in\mathbb{S}_+$ such that 

\begin{align}
\left[\begin{array}{c}I\\\left[\begin{array}{cc}A_{cl}&B_{cl}\\C_{cl}&D_{cl}\end{array}\right]\end{array}\right]^\top\left[\begin{array}{c|c}\begin{array}{cc}0&0\\0&-X_{11}\end{array}&\begin{array}{cc}P&0\\0&-X_{12}\end{array}\\\hline\begin{array}{cc}P&0\\0&-X_{12}^\top\end{array}&\begin{array}{cc}0&0\\0&-X_{22}\end{array}\end{array}\right]\left[\begin{array}{c}I\\\left[\begin{array}{cc}A_{cl}&B_{cl}\\C_{cl}&D_{cl}\end{array}\right]\end{array}\right]\preceq0.
\end{align}
The corresponding synthesis problem is then given, by including a parameter dependence on the multiplier, as $P_{H_{cl}}(X(\gamma);\gamma^2) $. The following remarks comment on the feasibility and the solution approaches for $P_{H_{cl}}(X(\gamma);\gamma^2)$.
\renewcommand*{\arraycolsep}{2pt}

\begin{remark}[Feasibility] 
Let
\begin{align*}
X^{-1} = \left[\begin{array}{cc}\widetilde{X}_{11}&\widetilde{X}_{12}\\\widetilde{X}_{12}^\top&\widetilde{X}_{22}\end{array}\right].
\end{align*}
It is well known ( see \cite{scherer2000linear} Chap. 4, \cite{scorletti1998improved} and references therein), using elimination lemma that $H_{cl}\in\text{IQC}(X)$ if and only if there exists $Q_1,Q_2\in\mathbb{S}_+$ such that the following hold
\begin{align}
\begin{array}{c}
\left[\begin{array}{cc}Q_1&I\\I&Q_2\end{array}\right]\succeq0,\\\\
U^\top\left[\begin{array}{c}I\\\left[\begin{array}{cc}A&B_1\\C_1&D_{11}\end{array}\right]\end{array}\right]^\top\left[\begin{array}{c|c}\begin{array}{cc}0&0\\0&-X_{11}\end{array}&\begin{array}{cc}Q_1&0\\0&-X_{12}\end{array}\\\hline\begin{array}{cc}Q_1&0\\0&-X_{12}^\top\end{array}&\begin{array}{cc}0&0\\0&-X_{22}\end{array}\end{array}\right]\left[\begin{array}{c}I\\\left[\begin{array}{cc}A&B_1\\C_1&D_{11}\end{array}\right]\end{array}\right]U\preceq0,\\\\
V^\top\left[\begin{array}{c}I\\\left[\begin{array}{cc}A&B_1\\C_1&D_{11}\end{array}\right]^\top\end{array}\right]^\top\left[\begin{array}{c|c}\begin{array}{cc}0&0\\0&-\widetilde{X}_{11}\end{array}&\begin{array}{cc}Q_2&0\\0&-\widetilde{X}_{12}\end{array}\\\hline\begin{array}{cc}Q_2&0\\0&-\widetilde{X}_{12}^\top\end{array}&\begin{array}{cc}0&0\\0&-\widetilde{X}_{22}\end{array}\end{array}\right]\left[\begin{array}{c}I\\\left[\begin{array}{cc}A&B_1\\C_1&D_{11}\end{array}\right]^\top\end{array}\right]V\preceq0,
\end{array}
\end{align}
where 
\begin{align*}
U = \left[\begin{array}{c}C_2\\D_{21}\end{array}\right]^\perp\text{ and }V = \left[\begin{array}{c}B_2\\D_{12}\end{array}\right]^\perp.
\end{align*}
\end{remark}

\begin{remark}[Solution]
There are many ways to compute the solution of $P_{H_{cl}}(X(\gamma);\gamma^2)$. One simple method proceeds as follows; First, use bisection \cite{boyd1989bisection} to minimize the sub-optimality level $\gamma$. Suppose the minimum value of $\gamma$ is achieved by $Q_1^*$ and $Q_2^*$. Compute non-singular $R_1$ and $R_2$ with $R_1R_2^\top=I-Q_2^*Q_1^*$, and determine $P$ by solving
\begin{align}
\left[\begin{array}{cc}Q_1^*&R_2\\I&0\end{array}\right]P = \left[\begin{array}{cc}I&0\\Q_2^*&R_1\end{array}\right].
\end{align}
Once $P$ has been determined, the condition $H_{cl}\in\text{IQC}(X)$ is linear in the parameters $A_c,B_c,C_c,D_c$ and a feasible point can easily be determined using available efficient convex optimization routines. Other regularization constraints can also be added to improve the conditioning of the resulting controller, or even determine the order of the controller. See \cite{scherer2000linear} for more details on the different ways to this computation.
\end{remark}

\section{Problem Definition}\label{Problem Definition}

\begin{figure}[ht]
\begin{center}
\begin{subfigure}[b]{0.45\linewidth}
\centering
\begin{tikzpicture}[line width = 1.5pt,auto,>=latex,node distance=1.6cm]
\tikzstyle{block}=[draw,rectangle,minimum width=2.5em, minimum height=2.5em]
\def\x{1.0}		
\def\y{0.2}		
\def\z{1.4}		
\node [block](M){$M$};
\node [block,above of=M](H){$\begin{matrix}
\,{H}_1 & &\\[-0.15em]&\ddots&\\[-0.15em]& & {H}_N \end{matrix}$};
\draw [<-] (M.east)+(0,\y) -- +(\x,\y) |- node[swap,pos=0.25]{$\mathbf{y}$} (H);
\draw [<-] (M.east)+(0,-\y) -- +(\z,-\y) node [anchor=west]{$\mathbf{w}$};
\draw [->] (M.west)+(0,\y) -- +(-\x,\y) |- node[pos=0.25]{$\mathbf{v}$} (H);
\draw [->] (M.west)+(0,-\y) -- +(-\z,-\y) node [anchor=east]{$\mathbf{z}$};
\end{tikzpicture}
\subcaption{Global System}
\label{fig_Global}
\end{subfigure}
\begin{subfigure}[b]{0.45\linewidth}
\centering
\begin{tikzpicture}[line width = 1.5pt,auto,>=latex,node distance=1.6cm]
\tikzstyle{block}=[draw,rectangle,minimum width=4.5em, minimum height=4.5em]
\def\x{1.0}		
\def\y{0.5}		
\def\z{1.4}		
\node [block](H){\Large $H_i$};
\node [draw,rectangle, minimum width=1em, minimum height=1em,thick](X) at (1.5em,-1.6em){$\mathbf{x}_i$};
\draw [->] (H.east)+(0,0) -- +(\z,0) node [anchor=west]{$\mathbf{y}_i$};
\draw [<-] (H.west)+(0,-\y) -- +(-\z,-\y) node [anchor=east]{$\mathbf{v}_i$};
\draw [<-] (H.west)+(0,+\y) -- +(-\z,+\y) node [anchor=east]{$\mathbf{u}_i$};
\end{tikzpicture}
\subcaption{Local System}
\label{fig_Local}
\end{subfigure}
\caption{ Large-scale Interconnected System.}
\label{fig_Interconnected}
\end{center}
\end{figure}
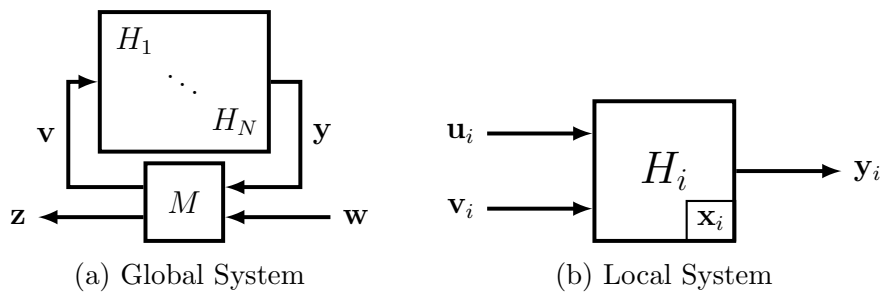

Figure~\ref{fig_Global} shows the interconnected system under consideration. The schematic for each $H_i$ block is shown in Figure~\ref{fig_Local}. The interconnection output vector $\mathbf{v}$ constitutes the exogenous inputs $\mathbf{v}_i$, $i=1,\hdots,N$ for the subsystems. On the other hand, the controlled outputs $\mathbf{y}_i$, $i=1,\hdots,N$ of all the subsystems constitute the interconnection input vector $\mathbf{y}$. The static matrix $M\triangleq\left[\begin{array}{cc}M_{11}&M_{12}\\M_{21}&M_{22}\end{array}\right]:\displaystyle \left[\begin{array}{c}\mathbf{y}\\\mathbf{w}\end{array}\right]\mapsto\left[\begin{array}{c}\mathbf{v}\\\mathbf{z}\end{array}\right]$ characterizes the interconnection topology. 
The term ``global" is used to describe the entire interconnected system as shown in Fig~\ref{fig_Global}, while ``local" refers to the individual subsystem as shown in Fig~\ref{fig_Local}. The goal of this section is to describe the synthesis problem for the global system in terms of the interconnection input-output pair $(\mathbf{w},\mathbf{z})$ and the local systems in terms of the input-output pairs $(\mathbf{v}_i,\mathbf{y}_i)$, $i=1,\hdots,N$. Moreover, to establish a relationship between the local and global syntheses problems such that the solution of the local synthesis problem provides at least a feasibility certificate for the global synthesis problem.

\begin{definition}[Global Problem]
Consider the interconnected system $G$ in Fig~\ref{fig_Global}, defined by the local subsystems $H_i$, $i=1,\hdots,N$ and the interconnection matrix $M\in\mathbb{R}^{(n_v+n_z)\times(n_y+n_w)}$. Given a global multiplier $W(\gamma)$, parametrized by $\gamma\in\mathbb{R}$, the global problem is defined as the synthesis problem $P_G(W(\gamma);\gamma^2)$ which is given by the optimization problem
\begin{align}
P_G:\left\{
\begin{array}{rl}
\text{min}&\gamma^2\\
\text{s.t}\\&\displaystyle\int_\mathbb{R}\left[\begin{array}{c}\widehat{\mathbf{w}}(j\omega)\\\widehat{\mathbf{z}}(j\omega)\end{array}\right]^*W(\gamma)\left[\begin{array}{c}\widehat{\mathbf{w}}(j\omega)\\\widehat{\mathbf{z}}(j\omega)\end{array}\right]d\omega\ge0,\hspace{2mm}\forall\hspace{2mm}\mathbf{w}\in\mathbb{L}_{2e}^{n_w}.
\end{array}
\right.
\end{align}
\end{definition}
\begin{remark}
It is straightforward to see how the definition above can be extended to encompass the syntheses problems given in \eqref{synthesis_square} and \eqref{synthesis_rectangle}. The form above is used for simplicity of exposition.
\end{remark}
\begin{definition}[Local Problem]
For the interconnected system $G$, we define a local problem as finite collection of syntheses problems associated with the subsystems $H_i$, $i = 1,\hdots, N$. In other words, given a finite collection of parametrized multipliers $\left\{X_i(\gamma)\right\}_{i=1}^N$, the local problem $P_L\left(\left\{X_i(\gamma)\right\}_{i=1}^N;\gamma^2\right)$ is given by 
\begin{align}
P_L\left(\left\{X_i(\gamma)\right\}_{i=1}^N;\gamma^2\right) = \left\{P_{H_i}\left(X_i(\gamma);\gamma^2\right)\right\}_{i=1}^N
\end{align}
where each $P_{H_i}\left(X_i(\gamma);\gamma^2\right)$ is the synthesis problem associated with each subsystem $H_i$.
\end{definition}
\begin{remark}
The local problem $P_L\left(\left\{X_i(\gamma)\right\}_{i=1}^N;\gamma^2\right)$ is feasible is if there exists at least one $\gamma_f\in\mathbb{R}$ such that all the syntheses problems $P_{H_i}\left(X_i(\gamma_f);\gamma_f^2\right)$, $i=1,\hdots,N$ are feasible. Moreover, $\gamma^* = P_L\left(\left\{X_i(\gamma)\right\}_{i=1}^N;\gamma^2\right)$ is called the solution of the local problem $P_L\left(\left\{X_i(\gamma)\right\}_{i=1}^N;\gamma^2\right)$ and is given by $\gamma^* = \max\left\{\bar{\gamma}_i\right\}_{i=1}^N$, where $\bar{\gamma}_i = P_{H_i}\left(X_i(\gamma);\gamma^2\right)$ is the solution of synthesis problem associated with the subsystem $H_i$.
\end{remark}
\begin{definition}[Global Admissibility]
The local problem $P_L\left(\left\{X_i(\gamma)\right\}_{i=1}^N;\gamma^2\right)$ is said to be globally admissible by the global problem $P_G(W(\gamma);\gamma^2)$ (or simply admissible) if the following condition holds for all parameter $\gamma\in\mathbb{R}$;
\begin{align}\label{GAC1}
\sum_{i=1}^N\left[\begin{array}{c}\mathbf{v}_i\\\mathbf{y}_i\end{array}\right]^\top X_i(\gamma)\left[\begin{array}{c}\mathbf{v}_i\\\mathbf{y}_i\end{array}\right] - \left[\begin{array}{c}\mathbf{w}\\\mathbf{z}\end{array}\right]^\top W(\gamma)\left[\begin{array}{c}\mathbf{w}\\\mathbf{z}\end{array}\right]\le0
\end{align}
for all $\left\{\mathbf{v}_i\right\}_{i=1}^N,\left\{\mathbf{y}_i\right\}_{i=1}^N,\mathbf{w},\text{ and }\mathbf{z}$.
\end{definition}
\begin{remark}
If a local problem $P_L$ is admissible by the global problem $P_G$, then the feasibility of $P_L$ is a certificate for the feasibility of $P_G$.
\end{remark}

\begin{theorem}[Global Admissibility Condition]
Given the static interconnection matrix $M\in\mathbb{R}^{(n_v+n_z)\times(n_y+n_w)}$, the local problem $P_L\left(\left\{X_i(\gamma)\right\}_{i=1}^N;\gamma^2\right)$, with each $X_i(\gamma)$ given by the conformal block,
\begin{align}
X_i(\gamma) = \left[\begin{array}{cc}X_i^{11}(\gamma)&X_i^{12}(\gamma)\\X_i^{{12}^\top}(\gamma)&X_i^{22}(\gamma)\end{array}\right]
\end{align}
is admissible by the global problem $P_G(W(\gamma);\gamma^2)$ if and only if the matrix inequality

\begin{align}\label{GAC2}
\left[\begin{array}{c}M\\I\end{array}\right]^\top\left[\begin{array}{cc|cc}X^{11}(\gamma)&&X^{12}(\gamma)&\\&-W_{22}(\gamma)&&-W_{12}^\top(\gamma)\\\hline {{X^{12}}^\top}^{\phantom{T}}(\gamma)&&X^{22}(\gamma)&\\&-W_{12}(\gamma)&&-W_{11}(\gamma)\end{array}\right]\left[\begin{array}{c}M\\I\end{array}\right]\preceq0,
\end{align}
where
\begin{align}
X^{jk}(\gamma) = \left[\begin{array}{ccc}X_1^{jk}(\gamma)&&\\&\ddots&\\&&X_N^{jk}(\gamma)\end{array}\right], \hspace{2mm}j,k\in\{1,2\},
\end{align}
is feasible for all $\gamma\in\mathbb{R}$.
\end{theorem}

\begin{proof}
The global admissibility condition in \eqref{GAC1} is equivalent to the inequality
\begin{align*}
\left[\begin{array}{c}\mathbf{v}\\\mathbf{y}\end{array}\right]^\top\left[\begin{array}{cc}X^{11}(\gamma)&X^{12}(\gamma)\\ {X^{12}}^\top(\gamma)&X^{22}(\gamma)\end{array}\right]\left[\begin{array}{c}\mathbf{v}\\\mathbf{y}\end{array}\right]
 -  \left[\begin{array}{c}\mathbf{w}\\\mathbf{z}\end{array}\right]^\top\left[\begin{array}{cc}W_{11}(\gamma)&W_{12}(\gamma)\\W_{12}^\top(\gamma)&W_{22}(\gamma)\end{array}\right]\left[\begin{array}{c}\mathbf{w}\\\mathbf{z}\end{array}\right]\le0,
\end{align*}
holding true for all $\mathbf{v}=[\mathbf{v}_1^\top\hdots\mathbf{v}_N^\top]^\top,\mathbf{y}=[\mathbf{y}_1^\top\hdots\mathbf{y}_N^\top]^\top$, $\mathbf{w}\text{ and }\mathbf{z}$ satisfying the interconnection constraint. After rearranging terms, it is straight forward to see that the global admissibility condition is equivalent to
\begin{align*}
\begin{array}{c}
\left[\begin{array}{c}\mathbf{v}\\\mathbf{z}\\\hline\mathbf{y}\\\mathbf{w}\end{array}\right]^\top\left[\begin{array}{cc|cc}X^{11}(\gamma)&&X^{12}(\gamma)&\\&-W_{22}(\gamma)&&-W_{12}^\top(\gamma)\\\hline {{X^{12}}^\top}^{\phantom{T}}(\gamma)&&X^{22}(\gamma)&\\&-W_{12}(\gamma)&&-W_{11}(\gamma)\end{array}\right]\left[\begin{array}{c}\mathbf{v}\\\mathbf{z}\\\hline\mathbf{y}\\\mathbf{w}\end{array}\right]\le0,\\\\
\left[\begin{array}{c}\mathbf{v}\\\mathbf{z}\end{array}\right] = M\left[\begin{array}{c}\mathbf{y}\\\mathbf{w}\end{array}\right]
\end{array}
\end{align*}
for all $\mathbf{v},\mathbf{z},\mathbf{y}\text{ and }\mathbf{w}$.
\begin{align*}
\Leftrightarrow\left[\begin{array}{c}M\\I\end{array}\right]^\top\left[\begin{array}{cc|cc}X^{11}(\gamma)&&X^{12}(\gamma)&\\&-W_{22}(\gamma)&&-W_{12}^T(\gamma)\\\hline {{X^{12}}^\top}^{\phantom{T}}(\gamma)&&X^{22}(\gamma)&\\&-W_{12}(\gamma)&&-W_{11}(\gamma)\end{array}\right]\left[\begin{array}{c}M\\I\end{array}\right]\preceq0.
\end{align*}
\end{proof}

Using the global admissibility condition, a relaxed optimization problem for the global synthesis problem can be written in terms of the local problem as follows

\begin{definition}[Relaxed Global Problem]
Consider the interconnected system $G$ in Fig~\ref{fig_Global}, defined by the local subsystems $H_i$, $i=1,\hdots,N$ and the interconnection matrix $M\in\mathbb{R}^{(n_v+n_z)\times(n_y+n_w)}$. Given a global multiplier $W(\gamma)$, parametrized by $\gamma\in\mathbb{R}$, the relaxed global problem is defined as the synthesis problem $\overline{P}_G(W(\gamma);\gamma^2)$ which is given by the optimization problem

\begin{align}\label{global_relaxed}
\overline{P}_G:\left\{
\begin{array}{rl}
\text{min}&\gamma^2\\
\text{s.t}\\
&\left[\begin{array}{c}M\\I\end{array}\right]^\top\left[\begin{array}{cc|cc}Z^{11}&&Z^{12}&\\&-W_{22}(\gamma)&&-W_{12}^\top(\gamma)\\\hline {{Z^{12}}^\top}^{\phantom{T}}&&Z^{22}&\\&-W_{12}(\gamma)&&-W_{11}(\gamma)\end{array}\right]\left[\begin{array}{c}M\\I\end{array}\right]\preceq0,\\\\
&\hspace{3mm}Z^{11}\succeq0,Z^{22}\preceq0,\\\\
&\hspace{3mm}H_i\in\text{IQC}(Z_i), \hspace{2mm}i=1,\hdots,N.
\end{array}
\right.
\end{align}
\end{definition}

\begin{remark}[Solution]
Similar to the approach used in \cite{meissen2014compositional}, a solution method based on the alternating direction method of multipliers (ADMM)\cite{boyd2011distributed} is considered. ADMM belongs to a category of operator splitting techniques. It solves convex optimization problems by breaking them into smaller pieces, each of which are then easier to handle. At each iterative step ($k$-th step), the solution implements the following Gauss-Seidel sweep:

\begin{enumerate}
	\item {$\left\{Z_i^{k+1}\right\}_{i=1}^N = \arg\min P_Z\left(\left\{X_i^k\right\}_{i=1}^N\right)$, where
	\begin{align}
  P_Z\triangleq\left\{
\begin{array}{rl}
\text{min}&\gamma^2 + \displaystyle\sum_{i=1}^N\left\|X_i^k-Z_i+V_i^k\right\|_F^2\\
\text{s.t}\\
&\left[\begin{array}{c}M\\I\end{array}\right]^\top\left[\begin{array}{cc|cc}Z^{11}&&Z^{12}&\\&-W_{22}(\gamma)&&-W_{12}^T(\gamma)\\\hline {{Z^{12}}^\top}^{\phantom{T}}&&Z^{22}&\\&-W_{12}(\gamma)&&-W_{11}(\gamma)\end{array}\right]\left[\begin{array}{c}M\\I\end{array}\right]\preceq0,\\\\
&\hspace{3mm}Z^{11}\succeq0,Z^{22}\preceq0.
\end{array}
\right.
	\end{align}
	}
	\item {$\left\{X_i^{k+1}\right\}_{i=1}^N = \arg\min P_X\left(\left\{Z_i^{k+1}\right\}_{i=1}^N\right)$, where
	\begin{align}
	P_X\triangleq\left\{
\begin{array}{rl}
\text{min}&\displaystyle\sum_{i=1}^N\left\|X_i-Z_i^{k+1}+V_i^k\right\|_F^2\\
\text{s.t}\\
&\hspace{3mm}H_i\in\text{IQC}(Z_i), \hspace{2mm}i=1,\hdots,N.
\end{array}
\right.
	\end{align}
	}
	\item{Dual Ascent Update
	\begin{align}
	V_i^{k+1} \leftarrow V_i^k + X_i^{k+1} - Z_i^{k+1}.
	\end{align}
	}
\end{enumerate}
It is clear that step 2 is a candidate for parallel implementation on a distributed processor. As a result, the ADMM solution steps described above can be interpreted as a method for decentralized control design. 
\end{remark}

Ultimately, we would like to translate the synthesis problem associated with the global problem to a corresponding set of syntheses problems associated with a globally admissible local problem. However, it is not clear  how this can be done at this point.  

Moreover, the relaxed global problem in \eqref{global_relaxed} incorporates the local synthesis problem. In order words, the solution of the global problem requires at least the feasibility conditions for the local problem. This makes it unsuitable for cases where one is interested in designing performance specifications for the local subsystems using only the interconnection and global objectives, without necessarily solving the local problem. 

These shortcomings are due to the general nature of the parameter dependence in the global admissibility condition above. It turns out that, by imposing a particular parametric structure on the global and local multipliers, the global admissibility condition can be made parameter-free. Hence, a much nicer decoupling of the global and local problems can be obtained. The following corollary examines the parametrization that is exploited in this paper.

\begin{corollary}[Well-posed Interconnection]
Suppose the interconnection defined by $M\triangleq\left[\begin{array}{cc}M_{11}&M_{12}\\M_{21}&M_{22}\end{array}\right]:\displaystyle \left[\begin{array}{c}\mathbf{y}\\\mathbf{w}\end{array}\right]\mapsto\left[\begin{array}{c}\mathbf{v}\\\mathbf{z}\end{array}\right]$ is such that $M_{12}^TM_{12}^{\phantom{T}}$ and $M_{21}^{\phantom{T}}M_{21}^T$ are invertible, then the global admissibility condition in \eqref{GAC2} is equivalent to
\begin{align}\label{GAC3}
\left[\begin{array}{c|c}M_{12}^\top X^{11}(\gamma)M_{12}-W_{11}(\gamma)&M_{12}^\top X^{12}(\gamma)-W_{12}(\gamma)M_{21}\\\hline {\left(M_{12}^\top X^{12}(\gamma)-W_{12}(\gamma)M_{21}\right)^\top}^{\phantom{T}}&X^{22}(\gamma)-M_{21}^\top W_{22}(\gamma)M_{21}\end{array}\right]\preceq0
\end{align}
\end{corollary}
\begin{proof}
Observe that since $M_{12}^\top M_{12}^{\phantom{T}}$ and $M_{21}^{\phantom{T}}M_{21}^\top$ are invertible, the interconnection matrix can be factorized as
\begin{align*}
\left[\begin{array}{cc}M_{11}&M_{12}\\M_{21}&M_{22}\end{array}\right] = \left[\begin{array}{cc}0&M_{12}\\M_{21}&0\end{array}\right]\left[\begin{array}{cc}I&M_{21}^\dagger M_{22}\\M_{12}^\dagger M_{11}&I\end{array}\right]
\end{align*}
where $M_{12}^\dagger\triangleq\left(M_{12}^\top M_{12}\right)^{-1}M_{12}^\top$ and $M_{21}^\dagger\triangleq M_{21}^\top\left(M_{21}M_{21}^\top\right)^{-1}$.

Pre- and post-multiplying the right hand side of the global admissibility condition in \eqref{GAC2} by $P^\top$ and $P$ respectively,
where
\begin{align*}
P = \left[\begin{array}{c|c}\left[\begin{array}{cc}I&M_{21}^\dagger M_{22}\\M_{12}^\dagger M_{11}&I\end{array}\right]^{-1}&\\\hline&\hspace{2mm}\displaystyle I^{\phantom{T^\top}}\hspace{2mm}\end{array}\right]
\end{align*}
yields
\begin{align*}
\left[\begin{array}{c}\left[\begin{array}{cc}0&M_{12}\\M_{21}&0\end{array}\right]\\I\end{array}\right]^\top\left[\begin{array}{cc|cc}X^{11}(\gamma)&&X^{12}(\gamma)&\\&-W_{22}(\gamma)&&-W_{12}^\top(\gamma)\\\hline {{X^{12}}^\top}^{\phantom{T}}(\gamma)&&X^{22}(\gamma)&\\&-W_{12}(\gamma)&&-W_{11}(\gamma)\end{array}\right]\left[\begin{array}{c}\left[\begin{array}{cc}0&M_{12}\\M_{21}&0\end{array}\right]\\I\end{array}\right]\preceq0,
\end{align*}
which after expanding and reordering the matrix elements becomes 
\begin{align*}
\left[\begin{array}{c|c}M_{12}^\top X^{11}(\gamma)M_{12}-W_{11}(\gamma)&M_{12}^\top X^{12}(\gamma)-W_{12}(\gamma)M_{21}\\\hline {\left(M_{12}^\top X^{12}(\gamma)-W_{12}(\gamma)M_{21}\right)^\top}^{\phantom{T}}&X^{22}(\gamma)-M_{21}^\top W_{22}(\gamma)M_{21}\end{array}\right]\preceq0
\end{align*}
\end{proof}
\begin{remark}
The requirements that $M_{12}^\top M_{12}^{\phantom{T}}$ and $M_{21}^{\phantom{T}}M_{21}^\top$ be invertible defines the \emph{well-posedness} of the interconnection namely; the pair $(\mathbf{v},\mathbf{y})$ is uniquely defined by $(\mathbf{w},\mathbf{z})$ and vice versa.
\end{remark}
\begin{remark}[Local Passivability]
Consider a passivity global objective, i.e $W = \left[\begin{array}{cc}0&I\\I&0\end{array}\right]$. It straight forward to see that, if the interconnection is well-posed as defined above, the interconnected system is locally \emph{passivable} if and only if there exists a block diagonal matrix $D$ such that
\begin{align}
M_{12}^\top D-M_{21}=0.
\end{align}
Equivalently, the interconnected system is locally \emph{passivable} if and only if the matrix $M_{12}\left(M_{12}^\top M_{12}\right)^{-1}M_{21}$ is block diagonal.
\end{remark}
\begin{corollary}[Quadratic Parametrization]
Suppose the global supply rate $W(\gamma)$ is quadratically depends on the parameter $\gamma$ as follows
\begin{align}\label{W_para}
W(\gamma) = \gamma^2 W_1 + 2\gamma W_2 + W_3.
\end{align}
If the local supply rates are equivalently parametrized as
\begin{align}\label{X_para}
X_i(\gamma) = \gamma^2 X_{i_1} + 2\gamma X_{i_2} + X_{i_3},
\end{align}
then given the static interconnection matrix $M\in\mathbb{R}^{(n_v+n_z)\times(n_y+n_w)}$, the local problem $P_L\left(\left\{X_i(\gamma)\right\}_{i=1}^N;\gamma^2\right)$ is admissible by the global problem $P_G(W(\gamma);\gamma^2)$ if and only if the linear matrix inequality
\begin{align}\label{GAC4}
Q_1^\top Y_LQ_1 - Q_2^\top Y_GQ_2\preceq0,
\end{align}
is feasible, where
\begin{align}
Q_1 &= I_{2\times2}\otimes\left[\begin{array}{cc}M_{12}&M_{11}\\0&I\end{array}\right],\hspace{2mm}
Q_2 = I_{2\times2}\otimes\left[\begin{array}{cc}I&0\\M_{22}&M_{21}\end{array}\right]\\\nonumber\\
Y_L(X) &= \left[\begin{array}{c|c}\begin{array}{ccc}X_{1_1}&&\\&\ddots&\\&&X_{N_1}\end{array}&\begin{array}{ccc}X_{1_2}&&\\&\ddots&\\&&X_{N_2}\end{array}\\\hline \begin{array}{ccc}{X_{1_2}^\top}^{\phantom{T}}&&\\&\ddots&\\&&X_{N_2}^\top\end{array}&\begin{array}{ccc}X_{1_3}&&\\&\ddots&\\&&X_{N_3}\end{array}\end{array}\right],\hspace{2mm}
Y_G(W) =\left[\begin{array}{c|c}W_1&W_2\\\hline {W_2^\top}^{\phantom{T}}&W_3\end{array}\right]
\end{align}
\end{corollary} 
\begin{proof}
First, observe that the global admissibility condition is equivalent to
\begin{align*}
\left[\begin{array}{cc}M_{12}&M_{11}\\0&I\end{array}\right]^\top\left[\begin{array}{cc}X^{11}(\gamma)&X^{12}(\gamma)\\X^{21}(\gamma)&X^{22}(\gamma)\end{array}\right]\left[\begin{array}{cc}M_{12}&M_{11}\\0&I\end{array}\right]\hspace{4cm}\\ - \left[\begin{array}{cc}I&0\\M_{22}&M_{21}\end{array}\right]^\top W(\gamma)\left[\begin{array}{cc}I&0\\M_{22}&M_{21}\end{array}\right]\preceq0,\hspace{2mm}\forall \gamma\in\mathbb{R}.
\end{align*}
Next, substituting the parametrization in \eqref{W_para} and \eqref{X_para} yields
\begin{align*}
\left[\begin{array}{c}\gamma I\\I\end{array}\right]^\top\left(Q_1^\top Y_LQ_1 - Q_2^\top Y_GQ_2\right)\left[\begin{array}{c}\gamma I\\I\end{array}\right]\preceq0,\hspace{2mm}\forall \gamma\in\mathbb{R}
\end{align*}
\begin{align*}
\Leftrightarrow Q_1^\top Y_LQ_1 - Q_2^\top Y_GQ_2\preceq0.
\end{align*}
\end{proof}
\begin{remark}
Given the interconnection matrix $M$ and the global supply rate parametrized as in \eqref{W_para}, solving \eqref{GAC4} for any feasible $Y_L$ satisfying $X^{11}\succeq0,\overline{X}^{11}\succeq0,X^{22}\preceq0$ yields a globally admissible local problem $P_L\left(\left\{X_i(\gamma)\right\}_{i=1}^N;\gamma^2\right)$. This process can be interpreted as a \emph{supply rate allocation technique}, in which performance specifications are prescribed for each subsystem such that the global performance is guaranteed for any quasi-convex performance level. This structure is exploited in the next section to analyze the achievable performance levels by the set of all globally admissible local problems.
\end{remark}
\begin{remark}
If the interconnection is well-posed as defined above, a more structured quadratic parametrization can be defined, viz.
\begin{align}\label{W_para2}
W(\gamma) = \left[\begin{array}{c|c}\gamma^2W_{11}+\overline{W}_{11}&2\gamma W_{12}+\overline{W}_{12}\\\hline {\left(2\gamma W_{12}+\overline{W}_{12}\right)^\top}^{\phantom{T}}& \overline{W}_{22}\end{array}\right],
\end{align}
and
\begin{align}\label{X_para2}
X_i(\gamma) = \left[\begin{array}{c|c}\gamma^2 X_i^{11}+\overline{X}_i^{11}&2\gamma X_i^{12}+\overline{X}_i^{12}\\\hline{{\left(2\gamma X_i^{12}+\overline{X}_i^{12}\right)}^\top}^{\phantom{T}}&\overline{X}_i^{22}\end{array}\right].
\end{align}
Thus the global admissible condition becomes
\begin{align}
\left[\begin{array}{c|c}\begin{array}{cc}M_{12}^\top X^{11}M_{12}-W_{11}&\\&0\end{array}&\begin{array}{cc}&M_{12}^\top X^{12}-W_{12}M_{21}\\{\left(M_{12}^\top X^{12}-W_{12}M_{21}\right)^\top}^{\phantom{T}}\end{array}\\\hline *&{\begin{array}{c|c}M_{12}^\top\overline{X}^{11}M_{12}-\overline{W}_{11}&M_{12}^\top\overline{X}^{12}-\overline{W}_{12}M_{21}\\\hline {\left(M_{12}^\top\overline{X}^{12}-\overline{W}_{12}M_{21}\right)^\top}^{\phantom{T}}&\overline{X}^{22}-M_{21}^\top\overline{W}_{22}M_{21}\end{array}}^{\phantom{T}}\end{array}\right]\preceq0
\end{align}
\end{remark}
\section{Localization}\label{Localization}

The term \emph{localization} is used to describe local problems obtained from a global problem in a way that certain global properties are ascertained.

\begin{definition}[Localization]
Given an interconnection defined by $M$ and a global problem $P_G\left(W(\gamma);\gamma^2\right)$, with $W(\gamma)$ given by \eqref{W_para}, the local problem $P_L\left(\left\{X_i(\gamma)\right\}_{i=1}^N;\gamma^2\right)$, with $X_i$ given by \eqref{X_para}, is called a localization of $P_G$ if the global admissibility condition in \eqref{GAC4} holds.
\end{definition}
\begin{definition}[Localization Gap]
The localization gap between the global problem $P_G\left(W(\gamma);\gamma^2\right)$ and a localization  $P_L\left(\left\{X_i(\gamma)\right\}_{i=1}^N;\gamma^2\right)$ is given by
\begin{align}
\left\|P_L-P_G\right\|\triangleq \sqrt{\gamma_L^2-\gamma_G^2},
\end{align}
where $\gamma_G$ and $\gamma_L$ are the respective solutions of the global and local problems, and are given by $\gamma_G=P_G\left(W(\gamma);\gamma^2\right)$ and $\gamma_L = P_L\left(\left\{X_i(\gamma)\right\}_{i=1}^N;\gamma^2\right) = \max\left\{\bar{\gamma}_i\right\}_{i=1}^N$, where $\bar{\gamma}_i = P_{H_i}\left(X_i(\gamma);\gamma^2\right)$ is the solution of synthesis problem associated with the subsystem $H_i$.
\end{definition}
\begin{definition}[Localization Distance]
Given an interconnection defined by $M$, a global problem $P_G\left(W(\gamma);\gamma^2\right)$ and a localization $P_L\left(\left\{X_i(\gamma)\right\}_{i=1}^N;\gamma^2\right)$, the localization distance between the global supply rate $W(\gamma)$ and the local supply rates $X\triangleq\left\{X_i(\gamma)\right\}_{i=1}^N$ is given by
\begin{align}
D(X,W) = \sigma_{max}\left(Q_1^\top Y_L(X)Q_1-Q_2^\top Y_G(W)Q_2\right).
\end{align}
\end{definition}
\begin{remark}[Closest Localization]
The closest localization is then obtained by solving the following minimization problem
\begin{align}
\begin{array}{rl}
\text{min}&D(X,W)\\
\text{s.t}&Q_1^\top Y_LQ_1 - Q_2^\top Y_GQ_2\preceq0,					
\end{array}
\end{align}
which is equivalent to the semi-definite program
\begin{align}
\begin{array}{rl}
\text{min}&-t\\
\text{s.t}&\\
					&tI-Q_1^\top Y_L(X)Q_1\preceq0\\\\
					&Q_1^\top Y_LQ_1 - Q_2^\top Y_GQ_2\preceq0,					
\end{array}
\end{align}
\end{remark}
\begin{remark}[Exact Localization]
The localization is said to be \emph{exact} if the global admissibility condition holds with equality. In this case the localization distance, hence the localization gap, is zero.
\end{remark}
\begin{lemma}\label{lemma1}
Let $P_L(X^\dagger;\gamma^2)$ be the closest localization of the global problem $P_G(W(\gamma);\gamma^2)$. Then
\begin{align}
Q_1^\top\left(Y_L(X)-Y_L(X^\dagger)\right)Q_1\preceq0,
\end{align}
for all localization $P_L(X;\gamma^2)$ of $P_G(W(\gamma);\gamma^2)$.
\end{lemma}
\begin{proof}
Suppose there exist a localization $P_L(X;\gamma^2)$ such that
\begin{align*}
Q_1^\top\left(Y_L(X)-Y_L(X^\dagger)\right)Q_1\succ0.
\end{align*}
It follows that
\begin{align*}
\lambda_{min}\left(Q_1^\top Y_L(X^\dagger)Q_1\right) - \lambda_{min}\left(Q_1^\top Y_L(X)Q_1\right)\le0
\end{align*}
which implies that $D(X,W)<D(X^\dagger,W)$, a contradiction.
\end{proof}
\begin{theorem}
Given an interconnected system and an associated global problem. If the closest localization is not feasible, then all feasible local problems are not globally admissible.
\end{theorem}
\begin{proof}
Let $\left\{X_i^\dagger\right\}_{i=1}^N$ be the set of multipliers associated with the closest localization. For any localization with multipliers $\left\{X_i\right\}_{i=1}^N$, it follows from Lemma~\ref{lemma1} that
\begin{align*}
Q_1^\top\left(Y_L(X)-Y_L(X^\dagger)\right)Q_1\preceq0,
\end{align*}
which implies that
\begin{align*}
\sum_{i=1}^N\left[\begin{array}{c}\mathbf{v}_i\\\mathbf{y}_i\end{array}\right]^\top\left(X_i(\gamma)-X_i^*(\gamma)\right)\left[\begin{array}{c}\mathbf{v}_i\\\mathbf{y}_i\end{array}\right]\le0,\hspace{2mm}\forall \gamma\in\mathbb{R}
\end{align*}
\begin{align*}
\Leftrightarrow\left[\begin{array}{c}\mathbf{v}_i\\\mathbf{y}_i\end{array}\right]^\top X_i(\gamma)\left[\begin{array}{c}\mathbf{v}_i\\\mathbf{y}_i\end{array}\right]\le\left[\begin{array}{c}\mathbf{v}_i\\\mathbf{y}_i\end{array}\right]^\top X_i^\dagger(\gamma)\left[\begin{array}{c}\mathbf{v}_i\\\mathbf{y}_i\end{array}\right],\hspace{2mm}\forall \gamma\in\mathbb{R}, i=1,\hdots,N.
\end{align*}
Now, $P_L\left(\left\{X_i^\dagger(\gamma)\right\}_{i=1}^N;\gamma^2\right)$ not feasible implies that there exists at least one $i\in\{1,\hdots,N\}$ such that
\begin{align*}
&\int_\mathbb{R}{\left[\begin{array}{c}\widehat{\mathbf{v}}_i\\\widehat{\mathbf{y}}_i\end{array}\right]^*X_i^\dagger(\gamma)\left[\begin{array}{c}\widehat{\mathbf{v}}_i\\\widehat{\mathbf{y}}_i\end{array}\right]d\omega}<0,\text{ for some } \mathbf{v}_i\in\mathbb{L}_{2e}\\
\Rightarrow&\int_\mathbb{R}{\left[\begin{array}{c}\widehat{\mathbf{v}}_i\\\widehat{\mathbf{y}}_i\end{array}\right]^*X_i(\gamma)\left[\begin{array}{c}\widehat{\mathbf{v}}_i\\\widehat{\mathbf{y}}_i\end{array}\right]d\omega}<0,
\end{align*}
which implies that $P_L\left(\left\{X_i(\gamma)\right\}_{i=1}^N;\gamma^2\right)$ not feasible.  
\end{proof}

The closest localization has the biggest set of multipliers over which the IQCs are defined. Thus, it's solution is a global lower bound on the solution of all globally admissible local problems. Consequently, the closest localization represents the most relaxed local problem that is globally admissible. It creates a nice way to transform global control synthesis problem into local syntheses problems for the subsystems in such a way that the resulting closed loop subsystems will interact through the interconnection to ensure that the feasibility of global problem is guaranteed.

It is noteworthy that no specific control design method is enforced. Only a somewhat unified description of of control design problem is made using IQC/Dissipativity. This is one of the merits of this approach because it allows for heterogeneous control synthesis for the subsystems. Moreover, addition requirements (e.g robustness, saturation etc) can be built in at the subsystems level. The only requirement is the resulting closed loop subsystem satisfies the IQC defined by the corresponding multiplier of the closest localization at the respective interconnection input-output port.

\subsection{Group localization}
\begin{definition}[Group]
A group is any finite collection of subsystems 
\begin{align}
 G = \left\{H_i\right\}_{\displaystyle i\in\left\{1,2,\hdots,N\right\}},
\end{align}
with group capacity $\#(G)\le N$, where the function $\#(.)$ returns the cardinality of its argument.
\end{definition}
\begin{remark}
Order does not matter in this definition. In other words, two groups are considered the same if the elements of any can be produced by rearranging the elements of the other. 
\end{remark}

Starting with two integers $\{N_g,\bar{N}\}<N$, the objective in this section is to identify disjointed groups $\left\{G_j\right\}_{j=1}^{N_g}$, with
\begin{align}
\begin{array}{l}
\displaystyle \sum_{j=1}^{N_g}{\#(G_j)}=N,\\
\#(G_j)\le\bar{N},\hspace{2mm} j\in\{1,2,\hdots N_g\},
\end{array}
\end{align}
and the corresponding multipliers $\left\{X_j\right\}_{j=1}^{N_g}$,  such that the local problem $P_L\left(\left\{X_j(\gamma)\right\}_{j=1}^{N_g};\gamma^2\right)$ is the closest localization. Here, for each localization, the $j$th synthesis problem is defined for the corresponding group which itself is a global problem with respect to its element subsystems. The underlying assumption in this section is that the capacity of each group is small enough that their respective synthesis problem can be solved without the need for any localization. This can be guaranteed by picking $\bar{N}$ small enough.

\begin{figure}[htbp]
	\centering
		\includegraphics[scale=0.75]{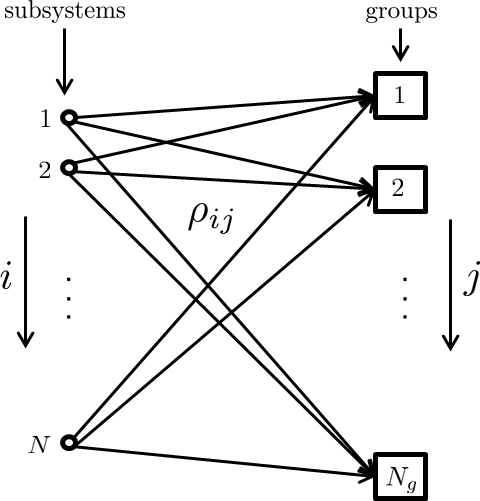}
	\caption{Group Localization}
	\label{fig:Group_Localization}
\end{figure}

The schematic for the group localization is shown in Figure~\ref{fig:Group_Localization}, where $\rho_{ij}$ denotes the degree of membership of subsystem $i$ in group $j$. For this problem, we take $\rho_{ij}\in\{0,\hspace{2mm}1\},i\in\{1,2,\hdots N\},j\in\{1,2,\hdots,N_g\}$. The following constraints are imposed on $\rho_{ij}$;
\begin{enumerate}
	\item {Each subsystem is allowed one and only group assignment. 
		\begin{align}\label{group_cons1}
		\sum_{j=1}^{N_g}\rho_{ij}=1, \hspace{2mm}\forall i\in\{1,2\hdots,N\}.
		\end{align}
	}
	\item{The number of subsystems in each group cannot exceed the group capacity $\bar{N}$.
	\begin{align}\label{group_cons2}
		\sum_{i=1}^N\rho_{ij}\le\bar{N}<N, \hspace{2mm}\forall j\in\{1,2\hdots,N_g\}.
	\end{align}
	}
\end{enumerate}

\begin{proposition}\label{relaxation_thm}
Let $P=P^\top\in\left\{0,\hspace{2mm}1\right\}^{N\times N}$ be a symmetric matrix of binary values that satisfies

\begin{itemize}
	\setlength\itemsep{0em}
	\item $\text{rank}(P)=N_g,$
	\item $P\succeq0,$
	\item $P_{ii} = 1, \hspace{2mm}i,j\in\{1,2,\hdots,N\}.$
\end{itemize}
Then $P$ admits the factorization $P=\rho\rho^\top$, where $\rho\in\left\{0,\hspace{2mm}1\right\}^{N\times N_g}$ is a matrix of binary numbers whose elements satisfy the constraints in \eqref{group_cons1} and \eqref{group_cons2} for some $\bar{N}<N$. Moreover, the nonzero singular values of $P$ are $\sigma(P)=\left\{\#(G_j)\right\}_{j=1}^{N_g}$.
\end{proposition}
\begin{remark}
The elements of matrix $P$ can be interpreted as
\begin{align}
P_{ij}=\left\{\begin{array}{rl}1&\text{ if the subsystems }i\text{ and }j\text{ belong to the same group}\\0&\text{ otherwise}\end{array}\right.
\end{align}
\end{remark}

\begin{proof}
It follows from all the properties of $P$ that there exists a permutation matrix $T$ such that
\begin{align}
P = T\left[\begin{array}{ccc}\mathcal{I}_{N_1}&&\\&\ddots&\\&&\mathcal{I}_{N_{N_g}}\end{array}\right]T^\top,
\end{align}
where $\mathcal{I}_{N_j}$ is an $N_j\times N_j$ matrix of $1$s, with $N_j\triangleq\#(G_j)$. Equivalently,
\begin{align}\label{similar}
P = T\left[\begin{array}{ccc}\mathbf{1}_{N_1}&&\\&\ddots&\\&&\mathbf{1}_{N_{N_g}}\end{array}\right]\left[\begin{array}{ccc}\mathbf{1}_{N_1}^\top&&\\&\ddots&\\&&\mathbf{1}_{N_{N_g}}^\top\end{array}\right]T^\top,
\end{align}
where $\mathbf{1}_{N_j}$ is an $N_j$-dimensional vector of $1$s. Thus, the first conclusion follows by taking
\begin{align}
\rho = T\left[\begin{array}{ccc}\mathbf{1}_{N_1}&&\\&\ddots&\\&&\mathbf{1}_{N_{N_g}}\end{array}\right].
\end{align}
Moreover, using \eqref{similar}, it is clear that $P$ is similar to $\left[\begin{array}{ccc}\mathcal{I}_{N_1}&&\\&\ddots&\\&&\mathcal{I}_{N_{N_g}}\end{array}\right]$. Thus
\begin{align}
\sigma\left(P\right)=\sigma\left(\left[\begin{array}{ccc}\mathcal{I}_{N_1}&&\\&\ddots&\\&&\mathcal{I}_{N_{N_g}}\end{array}\right]\right)=\left\{N_j\right\}_{j=1}^{N_g}=\left\{\#\left(G_j\right)\right\}_{j=1}^{N_g}.
\end{align}
\end{proof}

The result above demonstrates that the group membership function can be inferred from a positive definite binary matrix satisfying the states hypothesis. Thus, the group allocation problem is given by

\begin{align}
\begin{array}{rl}
\text{min}&D(P\circ X,W)\\
\text{s.t}&\\
					&Q_1^\top Y_L(P\circ X)Q_1 - Q_2^\top Y_G(W)Q_2\preceq0,\\
					&P\succeq0,\\
					&\text{rank}(P)=N_g,\\
					&P_{ij}\in\{0,\hspace{2mm}1\},\hspace{2mm}P_{ii}=1,
\end{array}
\end{align}

where the binary operator $\circ$ defined as
\begin{align}
P\circ X = \left[\begin{array}{cccc}P_{11}X_{11}&P_{12}X_{12}&\hdots&P_{1N}X_{1N}\\\vdots&\vdots&\ddots&\vdots\\P_{1N}X_{1N}^\top&P_{2N}X_{2N}^\top&\hdots&P_{NN}X_{NN}\end{array}\right].
\end{align}
Notice that the local supply rate is now allowed to be a full block as opposed to the diagonal block of the previous sub section, thereby giving more room to reduce conservatism. The optimization problem above is nonconvex and NP-hard. The details of the relaxation approaches and algorithms to solve the problem is left as future work. However, it is noted that, for a fixed $P$, the problem is convex and well-behaved. Thus a quick alternative approach considered in this present work is alternating minimization. First, the optimization is relaxed using the second part of proposition~\ref{relaxation_thm} and then regularized as follows;

\begin{align}
\begin{array}{rl}
\text{min}&D(P\circ X,W) + \Omega(P)\\
\text{s.t}&\\
					&Q_1^\top Y_L(P\circ X)Q_1 - Q_2^\top Y_G(W)Q_2\preceq0,\\
					&P\succeq0,\\
					&\sigma_{max}(P)\le \bar{N},\\
					&P_{ij}\in\left[0,\hspace{2mm}1\right],\hspace{2mm}P_{ii}=1,
\end{array}
\end{align}

where $\Omega(.)$ is any sparsity encouraging norm e.g $l_1$-norm\cite{candes2008enhancing}, trace-norms \cite{pong2010trace}, etc. Consequently, the alternating minimization algorithm is given in Algorithm~\ref{alg}.

\begin{algorithm}
\caption{Alternating minimization for group localization}\label{alg}
\begin{algorithmic}[1]
	\State{Set $X^0$ as such that
	\begin{align*}
	 Q_1^\top Y_L(X^0)Q_1 - Q_2^\top Y_G(W)Q_2\preceq0
	\end{align*}
	}
	\State{Initialize $k=0$ and $\varepsilon>0$ big enough}
		\While{$\varepsilon\ge\text{tol}$}\Comment{$\text{tol}$ is a specified convergence criterion}
			\State {Set $P^k$ by solving the minimization problem
								\begin{align*}
									\begin{array}{rl}
									\text{min}&D(P^k\circ X^k,W) + \Omega(P)\\
									\text{s.t}&\\
														&Q_1^\top Y_L(P^k\circ X^k)Q_1 - Q_2^\top Y_G(W)Q_2\preceq0,\\
														&P^k\succeq0,\\
														&\sigma_{max}(P^k)\le \bar{N},\\
														&P^k_{ij}\in\left[0,\hspace{2mm}1\right],\hspace{2mm}P^k_{ii}=1.
									\end{array}
									\end{align*}
						 }
			\State{Set $X^{k+1}$ by solving the minimization problem
									\begin{align*}
									\begin{array}{rl}
									\text{min}&D(P^k\circ X^{k+1},W)\\
									\text{s.t}&\\
														&Q_1^\top Y_L(P^k\circ X^{k+1})Q_1 - Q_2^\top Y_G(W)Q_2\preceq0.
									\end{array}
									\end{align*}
						 }
			\State $\varepsilon\gets |D(P^k\circ X^{k+1},W)-D(P^k\circ X^{k},W)|$
			\State $k\gets k+1$
		\EndWhile
		\State \textbf{return} $X^k$ and $P^k$
\end{algorithmic}
\end{algorithm}

\newpage

\section{Future Works}\label{Conclusion}
We conclude this paper by giving some interesting directions for future works.

\subsection*{Extended generalization} The approaches used in this paper can be extended to a more general class of problems by replacing the quadratic assumptions to a more general convex functions. An interesting example of such generalization is \emph{sum of squares} (SOS). SOS is vastly studied in literature and has very rich application in controls \cite{prajna2004new,papachristodoulou2002construction,sanchez1984fundamental}. As such, there exists a wealth of knowledge to build on. Direct application angles, for instance, are; more general description of syntheses problems using \emph{integral sum of squares constraint} and SOS parametrization of the multipliers.
\subsection*{Components specification for a large-scale design problem} One can describe components characteristics in a big design project as an IQC. Then use the methods developed in this paper to identify the set characteristics that best achieve the design objective which is also expressed as an IQC.
\subsection*{Detailed analysis of algorithms} Detailed algorithm design and analysis for the global admissibility problem and the group localization problems where not considered. Off-the-shelf  ADMM and alternating minimization algorithms were used. Design of specialized algorithms for the problem of the form considered in this paper is an interesting problem to consider.
\subsection*{Application} Applying a new technique to particular practical problems is always an interesting endeavor. The methods presented in this paper describe new ways of looking at decentralized control and performance specification problems.
\subsection*{Specialization to port-controlled Hamiltonian systems} Port-controlled Hamiltonian Systems (PCHS) are studied extensively \cite{schaft1999l2,PCHS1,PCHS2,PCHS3} as a generalization of the network modeling of physical systems with independent storage elements. Specializing the techniques in this paper to PCHS can give very useful insight into the design of control for large scale interconnected physical systems. Combined with techniques like graph separation\cite{ellis1987graph}, the results in this paper can be used to do control synthesis for large-scale multi-energy domain physical systems using graphical modeling tool like bondgraph\cite{karnopp2012system}
\subsection*{Dynamic interconnection} The problems considered in this paper assumed static interconnection. A natural extension is the inclusion of dynamics in the interconnection. It is not clear at this point if the results obtain so far will extend to dynamic interconnection. Moreover, what happens if the global and local IQCs are known and one is interested in obtaining an interconnection (static or dynamic) that guarantees the feasibility of the global admissibility condition is an interesting question to consider. 
\subsection*{Optimization counterpart of localization} Given a large scale convex optimization problem. Is there a way to obtain a set of smaller optimization problems in such a way that its solution is related to the solution of the original problem in a way we can quantify/design? If this is true, it might give an interesting extension to operator splitting problems.





\bibliographystyle{elsarticle-num}
\bibliography{references}







\end{document}